\documentclass[11pt,a4paper,reqno]{amsart}

\usepackage[english]{babel} 

\usepackage{fourier}
\usepackage[T1]{fontenc}

\usepackage{amssymb,amsrefs,amsmath,mathtools}
\usepackage{graphicx,nicefrac}
\usepackage[babel=true]{microtype} 
\usepackage[shortlabels]{enumitem}
\usepackage{overpic}
\usepackage{todonotes}
\usepackage{mathrsfs}
\usepackage{hyperref}
\usepackage{ourbib}
\usepackage{xcolor}
\usepackage{csquotes}
\calclayout

\setlist{leftmargin=4mm}
\newtagform{simple}{}{}{}

\hypersetup{
    colorlinks,
    linkcolor={red!50!black},
    citecolor={blue!50!black},
    urlcolor={blue!80!black}
}

\newcommand{\R}{\mathbb{R}}

\newcommand{\Z}{\mathbb{Z}}

\renewcommand{\S}{\mathbb{S}}
\newcommand{\eps}{\varepsilon}

\newcommand{\dV}{\,\mathsf{dV}}
\renewcommand{\phi}{\varphi}
\newcommand{\scal}{\mathrm{scal}}
\DeclareMathOperator{\Ric}{Ric}

\newcommand{\vol}{\mathsf{vol}}
\newcommand{\len}{\mathsf{len}}

\renewcommand{\O}{\mathsf{O}}
\newcommand{\dd}{\mathrm{d}}

\DeclareMathOperator{\id}{\mathrm{id}}
\DeclareMathOperator{\dist}{\mathrm{dist}}

\NewDocumentCommand{\charFun}{}{1}
\NewDocumentCommand{\Disk}{}{\mathrm{D}}
\NewDocumentCommand{\Ball}{}{\mathrm{B}}

\NewDocumentCommand{\Sphere}{}{\S}

\usepackage{thmtools}
\declaretheorem[name=Theorem]{theorem}

\declaretheorem[name=Corollary, numbered=no]{corollary*}

\declaretheorem[name=Theorem, numbered=no]{theorem*}
\declaretheorem[name=Question, numbered=no]{question}

\declaretheorem[name=Lemma, numberlike=theorem]{lemma}
\declaretheorem[name=Proposition, numberlike=lemma]{proposition}

\usepackage[noabbrev,capitalise]{cleveref} 
\crefdefaultlabelformat{#2\textup{#1}#3}

\allowdisplaybreaks

\begin{document}

\title[The degree condition in Llarull's theorem]{The degree condition in Llarull's theorem on scalar curvature rigidity} 
\author{Christian B\"ar}
\author{Rudolf Zeidler}

\address{Universit\"at Potsdam, Institut f\"ur Mathematik, 14476 Potsdam, Germany}
\email{\href{mailto:christian.baer@uni-potsdam.de}{christian.baer@uni-potsdam.de}}
\email{\href{mailto:rudolf.zeidler@uni-potsdam.de}{rudolf.zeidler@uni-potsdam.de}}
\urladdr{\href{https://www.math.uni-potsdam.de/baer/}{www.math.uni-potsdam.de/baer}}
\urladdr{\href{https://www.rzeidler.eu}{www.rzeidler.eu}}

\begin{abstract} 
Llarull's scalar curvature rigidity theorem states that a $1$-Lipschitz map \(f\colon M\to \S^n\) from a closed connected Riemannian spin manifold \(M\) with scalar curvature \(\scal\ge n(n-1)\) to the standard sphere \(\S^n\) is an isometry if the degree of \(f\) is nonzero.
We investigate if one can replace the condition \(\deg(f)\neq0\) by the weaker condition that \(f\) is surjective.
The answer turns out to be \enquote{no} for \(n\ge3\) but \enquote{yes} for \(n=2\).
If we replace the scalar curvature by Ricci curvature, the answer is \enquote{yes} in all dimensions.
\end{abstract}



\date{\today}

\maketitle

\section{Introduction} 

In 1998 Llarull published the following rigidity theorem for scalar curvature \cite{Ll}*{Theorem~B}.

\begin{theorem*}[Llarull]\label{thm.Llarull}
Let $n\ge2$ and let \((M,g)\) be an $n$-dimensional closed connected Riemannian spin manifold with $\scal_g\ge n(n-1)$ and let $f\colon M\to \S^n$ be a smooth $1$-Lipschitz map with $\deg(f)\ne0$. 

Then $f$ is an isometry.
\end{theorem*}

Applying this to the special case $M=S^n$ and $f=\id$, we obtain that the only Riemannian metric $g$ on the sphere $S^n$ with $\scal_g\ge n(n-1)=\scal_{g_{\mathsf{std}}}$ and $g\ge g_{\mathrm{std}}$ is the standard round metric $g_{\mathrm{std}}$ itself \cite{Ll}*{Theorem~A}.
The assumptions in Llarull's theorem can be weakened in various ways:
\begin{itemize}[label=$\triangleright$]
\item 
smoothness of $f$ can be dropped \cites{B,CHS,CHSS,LT};
\item
if $n\ge3$, then $1$-Lipschitzness of $f$ can be replaced by the weaker condition that the map induced by $df$ on $2$-vectors in $\bigwedge^2TM$ is nonexpanding \cite{Ll}*{Theorem~C};
\item
the metric $g$ may have lower regularity than smoothness \cites{CHS,CHSS,LT};
\item
the target manifold $\S^n$ can be replaced by certain other manifolds \cite{GS};
\item
the scalar curvature condition and the contraction property of $f$ can be combined into a single condition \cite{Li};
\item 
the manifold $M$ may have non-empty boundary if one imposes conditions on its mean curvature \cites{BBHW,CZ,HKKZ,HLS,Lo};
\item
the manifold $M$ may be noncompact with an incomplete metric and the target manifold $\S^n$ with two antipodal punctures \cites{BBHW,HKKZ,HLS}.
\end{itemize}

In the present note we investigate to what extent the condition $\deg(f)\ne0$ can be relaxed.
One certainly cannot just drop it, as a constant map $f$ would immediately lead to a counterexample.
One needs a condition that ensures that the map $f$ \enquote{wraps around the sphere} in a certain sense.
Therefore, we ask the following question.

\begin{question}
Does Llarull's theorem still hold if we replace the condition $\deg(f)\ne0$ by the weaker condition that $f$ is surjective?
\end{question}

We will see that the answer is \enquote{yes} in the case of dimension $n=2$ and that it is \enquote{no} in dimensions $n\ge3$.

Let $\S^n=(S^n,g_{\mathrm{std}})$ denote the $n$-dimensional sphere with the standard metric of constant sectional curvature $1$ and by $\S_\delta^n=(S^n,\delta\cdot g_{\mathrm{std}})$ the $n$-sphere with radius $\delta>0$ and constant curvature $\frac{1}{\delta^2}$.

Let $V$ and $W$ be two $n$-dimensional oriented Euclidean vector spaces.
For any linear map $A\colon V\to W$ we denote by $\bigwedge^n A\colon \bigwedge^n V\to \bigwedge^n W$ the induced linear map on the maximal exterior powers.
Since $\bigwedge^n V$ and $\bigwedge^n W$ are one-dimensional, oriented and carry an induced scalar product, there is a canonical identification $\bigwedge^n V\cong \R \cong \bigwedge^nW$.
With respect to this identification, we can view $\bigwedge^n A$ as acting by multiplication with a real number which we denote by $\det(A)$.
If $0\le \mu_1 \le \dots \le \mu_n$ are the singular values of $A$, then $\lvert \det(A)\rvert =\mu_1\cdots \mu_n$.
Note that the quantity \(\lvert \det(A)\rvert\) is independent of the chosen orientations on \(V\) and \(W\).

We start with the following result showing that the analogous question for Ricci curvature has a positive answer.

\begin{theorem}
\label{Ricci_volume_llarull}
Let $M$ be a connected closed Riemannian manifold of dimension $n\ge2$ with Ricci curvature \(\Ric \geq n-1\).
Let $f\colon M\to \S^n$ be a surjective Lipschitz map such that $\lvert \det(\dd_xf)\rvert \le1$ for almost all $x\in M$.

Then $M$ is isometric to $\S^n$ and $f$ is a bi-Lipschitz homeomorphism.
If, moreover, $f$ is $1$-Lipschitz, then it is an isometry.
\end{theorem}
Note that by \cite{MS}*{Theorem~8}, a metric isometry between smooth Riemannian manifolds is a smooth Riemannian isometry.

Also observe that the \(1\)-Lipschitz case of \cref{Ricci_volume_llarull} is a consequence of Bishop--Gromov volume comparison and the well-known \enquote{Lipschitz-volume rigidity} principle (see e.g.~\cite{burago-ivanov}*{Lemma~9.1}).
Our proof of \cref{Ricci_volume_llarull} relies on a variant of the latter tailored to our purposes.
In particular, this answers the question from above positively in dimension $n=2$, including a treatment of not necessarily smooth Lipschitz maps which are merely nonexpanding on \(2\)-vectors:

\begin{corollary*}
\label{thmA}
Let $M$ be a connected closed Riemannian manifold of dimension $n=2$ with Gauss curvature $K\ge1$. 
Let $f\colon M\to \S^2$ be a surjective Lipschitz map such that the map induced by $\dd_x f$ on $2$-vectors in $\bigwedge^2TM$ is nonexpanding for almost all $x\in M$.

Then $M$ is isometric to $\S^2$ and $f$ is a bi-Lipschitz homeomorphism.
If, moreover, $f$ is $1$-Lipschitz, then it is an isometry.
\end{corollary*}

In contrast, for scalar curvature the situation in higher dimensions is completely different as our second theorem shows.

\begin{theorem}
\label{thmB}
Let $M$ and $N$ be connected closed smooth manifolds of dimension $n\ge3$.
Assume that $M$ admits Riemannian metrics with positive scalar curvature and let $N$ carry any Riemannian metric $g_N$.

Then for each $S_0>0$ and $\eps>0$ there exists a Riemannian metric $g_M$ on $M$ with $\scal_{g_M}\ge S_0$ and a smooth surjective $\eps$-Lipschitz map $f\colon (M,g_M)\to (N,g_N)$.
\end{theorem}

Applying \cref{thmB} with $N=\S^n$, $S_0=n(n-1)$ and $\eps=1$ shows that the answer to the question is negative in dimensions at least $3$.
Note that the spin condition plays no role in our considerations as no Dirac operator methods will be used.

\subsection*{Acknowledgments} 
It is our pleasure to thank Bernhard Hanke and Thomas Schick for asking the question at various occasions as well as Georg Frenck and Alexander Lytchak for stimulating discussions.

\medskip

{\footnotesize\noindent Funded by the European Union (ERC Starting Grant 101116001 – COMSCAL). Views and opinions expressed are however those of the author(s) only and do not necessarily reflect those of the European Union or the European Research Council. Neither the European Union nor the granting authority can be held responsible for them.

\noindent Funded by the Deutsche Forschungsgemeinschaft (DFG, German Research Foundation) – Project numbers
390685587, 441731261, 523079177.}

\section{Proof of Theorem~\ref{Ricci_volume_llarull}}
\label{sec.A}

For the proof of \cref{Ricci_volume_llarull}, we need a version of Lipschitz-volume rigidity.
We start with the following fundamental lemma on Lipschitz homeomorphisms.

\begin{lemma}\label{lip_homeomorphism}
Let \(M\) and \(N\) be $n$-dimensional connected closed Riemannian manifolds.
Let $f\colon M\to N$ be a homeomorphism which is a Lipschitz map with Lipschitz constant \(L\ge 1\).
Assume that at almost all points of $M$, the differential $\dd f$ of $f$ satisfies $\lvert \det(\dd f)\rvert=1$.

Then \(f^{-1}\colon N\to M\) is a Lipschitz map with Lipschitz constant $\le L^{n-1}$.
\end{lemma}

For any subset $A$ of a metric space $(X,\dist)$ and positive number $r$ we denote by $\Ball_r(A)=\{x\in X\mid \dist(x,A)<r\}$ the open $r$-neighborhood of $A$ in $X$.
In particular, if $A=\{x\}$ is a point, $\Ball_r(x)$ is the open $r$-ball about $x$ in $X$.

By $\omega_k$ we denote the volume of the $k$-dimensional Euclidean unit ball.

\begin{proof}[Proof of \cref{lip_homeomorphism}]
Fix $\eps>0$.
We choose $r_0>0$ so small that 
\begin{enumerate}[label=$\triangleright$]
\item $r_0$ is smaller than the convexity radii of $M$ and $N$,
\item the singular values of the differential $\dd_v\exp_p$ of the Riemannian exponential map of $M$ lie between $1-\eps$ and $1+\eps$ whenever $v\in T_pM$ with $\lvert v\rvert<r_0$ and $p\in M$.
\end{enumerate}
Let \(x \in M\) and \(y = f(x) \in N\).
Choose $r_1\in(0,r_0]$ such that $\Ball_{r_1}(x)\subseteq f^{-1}(\Ball_{r_0}(y))$.
Since $f$ is a homeomorphism, $f(\Ball_{r_1}(x))$ is an open neighborhood of $y$ in $N$.
Choose $r_2\in(0,r_1)$ such that \(\Ball_{r_2}(y)\subseteq f(\Ball_{r_1}(x))\), see \cref{fig.3}.
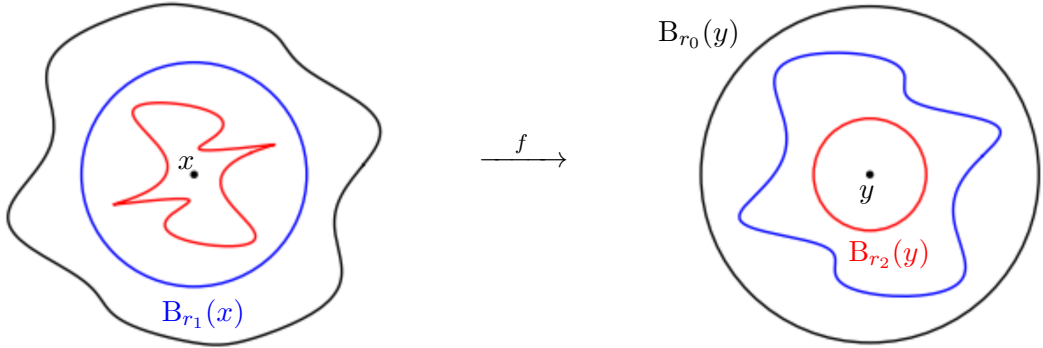
\begin{figure}[h]
\centering
\begin{overpic}[width=1\textwidth]{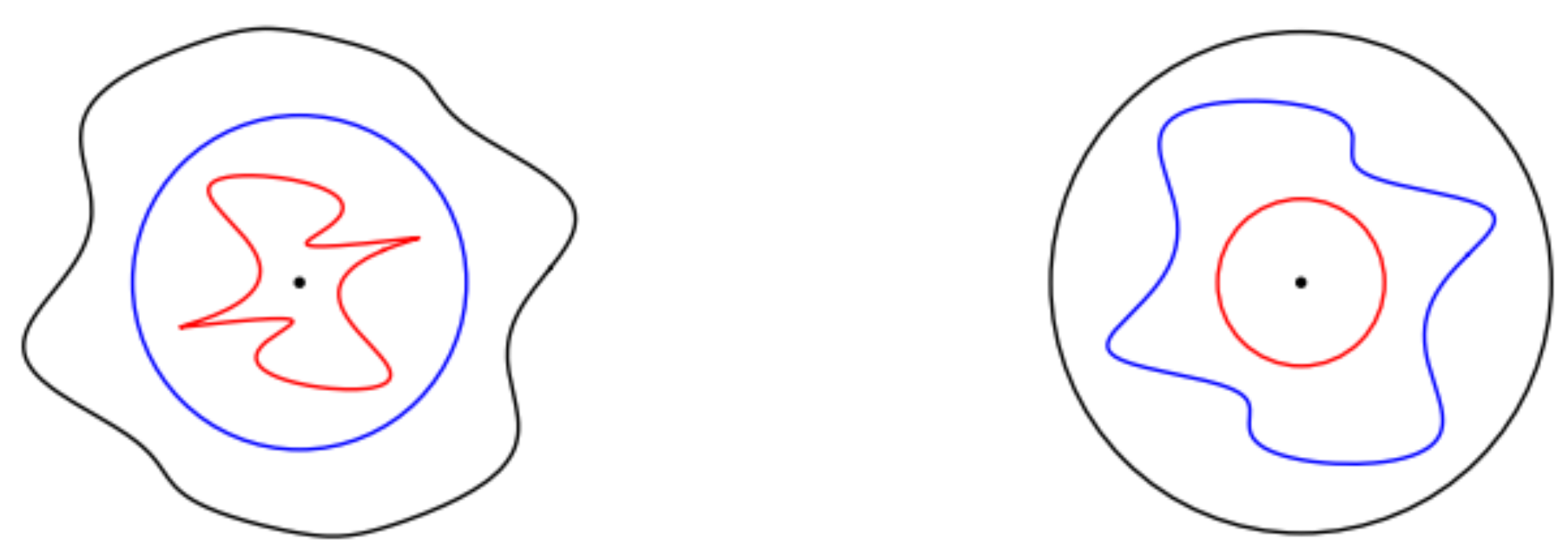}
\put(17.5,17.8){$x$}
\put(82,15){$y$}
\put(81,9){$\textcolor{red}{\Ball_{r_2}(y)}$}
\put(16,3.3){$\textcolor{blue}{\Ball_{r_1}(x)}$}
\put(63,30){$\textcolor{black}{\Ball_{r_0}(y)}$}
\put(46,18){$\xrightarrow{\quad f\quad}$}
\end{overpic}
\caption{Choice of neighborhoods}
\label{fig.3}
\end{figure}

Now let $y_1,y_2 \in \Ball_{r_2}(y)$ and put $x_j \coloneq f^{-1}(y_j)$.
Denote the shortest geodesic from $y_1$ to $y_2$ by $\gamma\colon [0,1]\to N$.
Since $r_2$ is smaller than the convexity radius of $N$, the curve $\gamma$ is entirely contained in $\Ball_{r_2}(y)$.
Thus, $c \coloneq f^{-1}\circ\gamma\colon [0,1]\to M$ is a continuous curve connecting $x_1$ and $x_2$ which is contained in $\Ball_{r_1}(x)$.

By the Lipschitz property of $f$, we have $f(B_r(c([0,1])))\subseteq \Ball_{Lr}(\gamma([0,1]))$ for any $r>0$.
Since $f$ is volume preserving, this implies 
\begin{equation}
\vol_M(\Ball_r(c([0,1]))) \le \vol_N(\Ball_{Lr}(\gamma([0,1]))).
\label{eq.lip1}
\end{equation}

Let $v_1,v_2\in \Ball_{r_1}(0)\subseteq T_x M$ be the tangent vectors with $\exp_{x}(v_j)=x_j$.
We obtain a second curve $\bar{c}\colon [0,1]\to M$ connecting $x_1$ and $x_2$ by putting $\bar{c}(t) \coloneq \exp_x(\ell(t))$ where $\ell(t)=tv_2+(1-t)v_1$ is the straight line connecting $v_1$ and $v_2$.
This curve is also contained in $\Ball_{r_1}(x)$, see \cref{fig.2}.
\begin{figure}[h]
\centering
\begin{overpic}[width=1\textwidth]{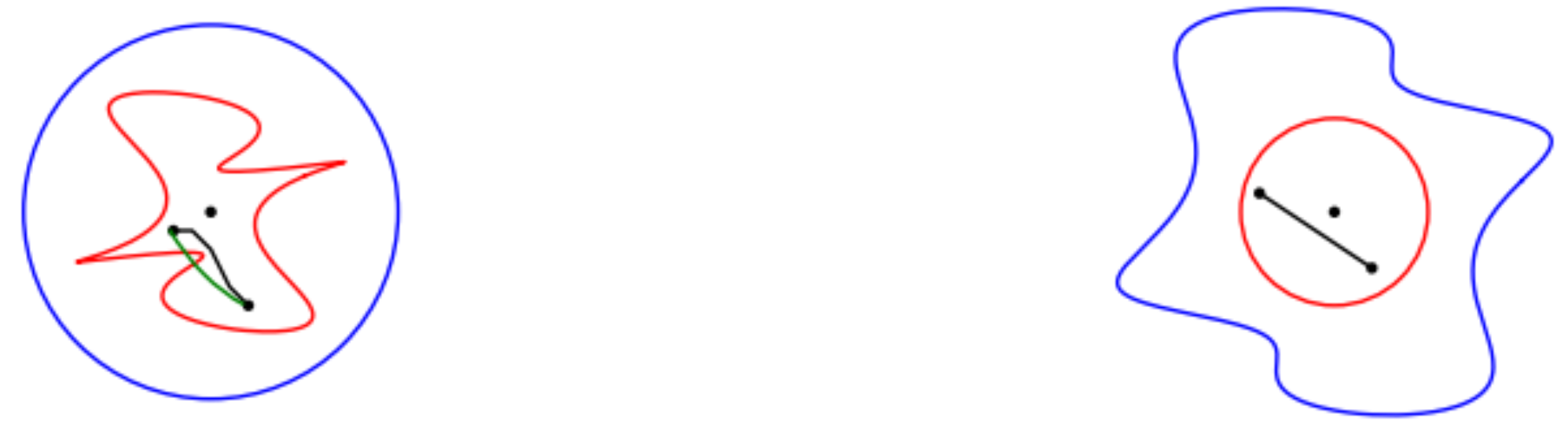}
\put(7,14){$x_1$}
\put(16.2,7){$x_2$}
\put(81,16){$y_1$}
\put(87,11.7){$y_2$}
\put(82,10.5){$\gamma$}
\put(14.2,10.5){$c$}
\put(11.5,7.5){\textcolor{green!50!black}{$\bar{c}$}}
\put(42,15){$\xrightarrow{\quad f\quad}$}
\end{overpic}
\caption{Connecting curves}
\label{fig.2}
\end{figure}
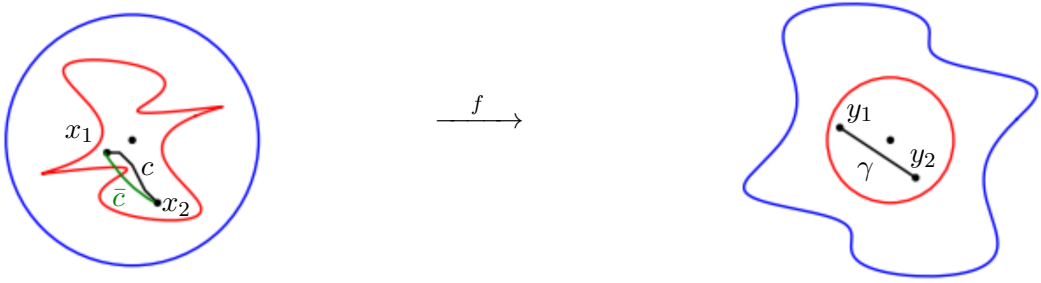

For subsets of $T_xM$ we measure distances and volumes with respect to the Euclidean metric given by the Riemannian metric at $x$.
\smallskip

\emph{Claim:} For all $r>0$ we have 
\begin{equation}
\vol_{T_xM}(\Ball_r(\ell([0,1]))) \le \vol_{T_xM}(\Ball_r(\exp_x^{-1}\circ c([0,1]))) .
\label{eq.lip2}
\end{equation}

\emph{Proof of claim:}
Denote the affine hyperplane in $T_xM$ through $v_1=\ell(0)$ which is perpendicular to the line $\ell([0,1])$ by $H$ and let $H_0$ be the corresponding vector subspace of $T_xM$.
Denote the open $r$-ball in $H_0$ about the origin by $\Disk_r$.
The tube $\Ball_r(\ell([0,1]))$ can be explicitly written as
\[
\Ball_r(\ell([0,1])) 
= 
Y_1 \cup \ell([0,1])\times\Disk_r \cup Y_2
\]
where $Y_1$ and $Y_2$ are the two \enquote{outer} open half-balls of radius $r$ at the endpoints $v_1$ and $v_2$ of the line $\ell([0,1])$.
We define the subset $X\subset T_xM$ by attaching the $(n-1)$-ball $\Disk_r$ to each point of the curve $\exp_x^{-1}\circ c$.
More precisely,
\[
X \coloneq Y_1 \cup \bigcup_{t\in[0,1]}(\Disk_r+\exp_x^{-1}(c(t))) \cup Y_2,
\]
see \cref{fig.4}.
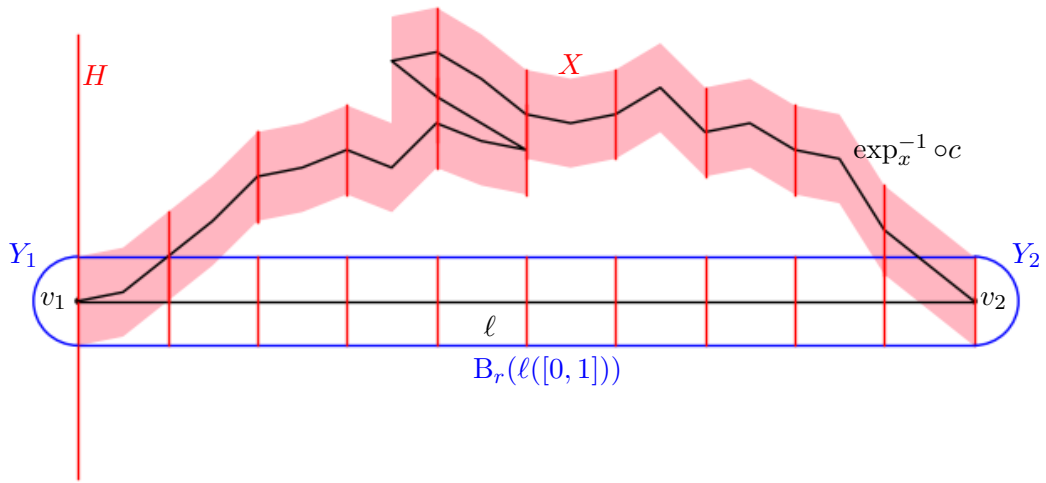
\begin{figure}[h]
\centering
\begin{overpic}[width=1\textwidth]{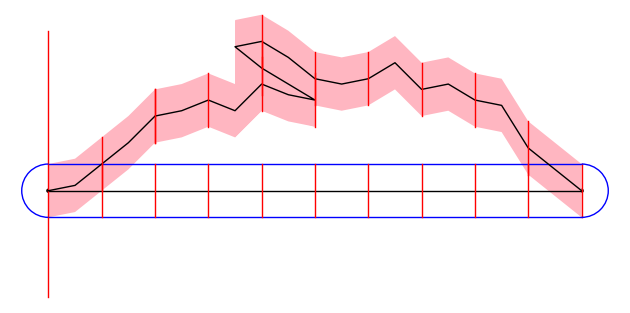}
\put(4,19){$v_1$}
\put(93,19){$v_2$}
\put(8,40){\textcolor{red}{$H$}}
\put(53,41){\textcolor{red}{$X$}}
\put(81,33){$\exp_x^{-1}\circ c$}
\put(45,12){\textcolor{blue}{$\Ball_r(\ell([0,1]))$}}
\put(46,16){$\ell$}
\put(1,23){\textcolor{blue}{$Y_1$}}
\put(96,23){\textcolor{blue}{$Y_2$}}
\end{overpic}
\caption{Construction of the set $X$}
\label{fig.4}
\end{figure}
Each point of $X$ is at distance $<r$ from some point of $\exp_x^{-1}\circ c([0,1])$, hence we have 
\begin{equation}
X\subseteq \Ball_r(\exp_x^{-1}\circ c([0,1])).
\label{eq.claim1}
\end{equation}
Moreover, for each $t\in[0,1]$ the hyperplane $H_0+\ell(t)$ intersects $\exp_x^{-1}\circ c([0,1])$ because $\exp_x^{-1}\circ c([0,1])$ is connected.
Therefore, the coarea formula yields
\[
\vol_{T_xM}(X)
\ge
\vol_{T_xM}(Y_1) + \len_{T_xM}(\ell)\cdot\vol_{H_0}(D_r) + \vol_{T_xM}(Y_2)
=
\vol_{T_xM}(\Ball_r(\ell([0,1]))) .
\]
Together with \eqref{eq.claim1} this proves the claim.
\hfill\checkmark

\smallskip

For sufficiently small $r>0$, we have $\Ball_r(\ell([0,1]))\subseteq \Ball_{r_1}(0)$, $\Ball_r(\exp_x^{-1}\circ c([0,1]))\subseteq \Ball_{r_1}(0)$, $\Ball_r(c([0,1]))\subseteq \Ball_{r_1}(x)$, and $\Ball_r(\bar{c}([0,1]))\subseteq \Ball_{r_1}(x)$.
The bounds on the singular values of the differential of $\dd\exp_x$ imply $\Ball_{(1-\eps)r}(\bar{c}([0,1]))\subseteq \exp_x(\Ball_r(\ell([0,1])))\subseteq \Ball_{(1+\eps)r}(\bar{c}([0,1]))$ and hence
\begin{align}
(1-\eps)^n\,\vol_M\big(\Ball_{(1-\eps)r}(\bar{c}([0,1]))\big)
&\leq 
\vol_{T_xM}\big(\Ball_r(\ell([0,1]))\big) \notag\\
&\leq 
(1+\eps)^n\,\vol_M\big(\Ball_{(1+\eps)r}(\bar{c}([0,1]))\big) .
\label{eq.lip3}
\end{align}
Similarly, we get
\begin{align}
(1-\eps)^n\,\vol_M\big(\Ball_{(1-\eps)r}(c([0,1]))\big)
&\leq 
\vol_{T_xM}\big(\Ball_r(\exp_x^{-1}\circ c([0,1]))\big) \notag\\
&\leq 
(1+\eps)^n\,\vol_M\big(\Ball_{(1+\eps)r}(c([0,1]))\big) .
\label{eq.lip4}
\end{align}
Combining these estimates, we find
\begin{align}
\vol_N(\Ball_{L(1+\eps)r}(\gamma([0,1])))
&\stackrel{\eqref{eq.lip1}}{\geq}
\vol_M(\Ball_{(1+\eps)r}(c([0,1]))) \notag\\
&\stackrel{\eqref{eq.lip4}}{\geq}
(1+\eps)^{-n}\,\vol_{T_xM}\big(\Ball_r(\exp_x^{-1}\circ c([0,1]))\big) \notag\\
&\stackrel{\eqref{eq.lip2}}{\geq}
(1+\eps)^{-n}\,\vol_{T_xM}(\Ball_r(\ell([0,1]))) \notag\\
&\stackrel{\eqref{eq.lip3}}{\geq}
\big[\tfrac{1-\eps}{1+\eps}\big]^n\,\vol_{M}(\Ball_{(1-\eps)r}(\bar{c}([0,1]))) .
\label{eq.lip5}
\end{align}
The power series expansion of the volume of thin tubes around a smooth curve yields
\begin{align*}
\vol_M(\Ball_{\rho}(\bar{c}([0,1]))) 
=
\len_M(\bar{c})\cdot\omega_{n-1}\cdot \rho^{n-1} + \O(\rho^{n+1})
\quad \text{ as }\rho\searrow 0 \, ,
\end{align*}
and similarly for the curve $\gamma$ in $N$.
Dividing \eqref{eq.lip5} by $\omega_{n-1}\cdot r^{n-1}$ and letting $r$ tend to $0$, we obtain
\[
L^{n-1}\,(1+\eps)^{n-1}\,\len_N(\gamma) \geq \big[\tfrac{1-\eps}{1+\eps}\big]^n\,(1-\eps)^{n-1}\,\len_M(\bar{c}).
\]
This yields
\[
L^{n-1}\cdot\dist(y_1,y_2)
=
L^{n-1}\cdot\len_N(\gamma)
\ge
\big[\tfrac{1-\eps}{1+\eps}\big]^{2n-1}\,\len_M(\bar{c})
\ge
\big[\tfrac{1-\eps}{1+\eps}\big]^{2n-1}\,\dist(x_1,x_2).
\]
This shows that $f^{-1}$ is Lipschitz with Lipschitz constant $\le L^{n-1}\cdot\big[\tfrac{1+\eps}{1-\eps}\big]^{2n-1}$ on the ball $\Ball_{r_2}(y)$.
Since $y$ is arbitrary, $f^{-1}$ is locally Lipschitz on $N$ with this global bound on the (local) Lipschitz constant.
Thus $f^{-1}$ is globally $\big(\big[\tfrac{1+\eps}{1-\eps}\big]^{2n-1}\cdot L^{n-1}\big)$-Lipschitz on $N$.
The limit $\eps\searrow 0$ concludes the proof.
\end{proof}

\begin{proposition}\label{1lip_injectivity}
Let \(M\) and \(N\) be connected closed Riemannian manifolds of the same dimension \(n \geq 1\).
Let \(f \colon M \to N\) be an $L$-Lipschitz map such that \(\vol(M) \le \vol_N(f(M))\) and \(\lvert\det(\dd f)\rvert\le1\) almost everywhere.

Then \(f\) is a homeomorphism and $f^{-1}\colon N\to M$ is $L^{n-1}$-Lipschitz.
In particular, if $f$ is $1$-Lipschitz, it is an isometry.
Moreover, for every measurable subset \(A \subseteq M\), we have \(\vol_N(f(A)) = \vol_M(A)\).
\end{proposition}

\begin{proof}
Let \(A \subseteq M\) be measurable. 
By the coarea formula for Lipschitz maps, the fiber \(f^{-1}(y)\) is finite for almost all \(y \in N\), and we obtain 
\begin{align*}
\vol_M(A) 
&\geq 
\int_{M} \charFun_A \lvert\det(d_x f)\rvert \dV_M \\
&= 
\int_N \sum_{x \in f^{-1}(y)} \charFun_A(x) \dV_N(y) \\
&= 
\int_N \#(f^{-1}(y) \cap A) \dV_N(y) \\
&\geq 
\vol_N(f(A)).
\end{align*}
Applying this to \(A = M\) and using the assumption \(\vol(M) \le \vol_N(f(M))\), we obtain equality in each inequality, which implies that \(\lvert \det(\dd_x f) \rvert = 1\) for almost all \(x \in M\) and \(\# f^{-1}(y) = 1\) for almost all \(y \in f(M)\).
Thus we actually have equality in the previous chain of inequalities for any measurable subset \(A \subseteq M\) which shows that \(\vol_M(A) = \vol_N(f(A))\).

In particular, the set of injectivity points has full measure in $M$.
Here, by \enquote{injectivity point}, we mean a point \(x \in M\) which is the unique preimage of its image under $f$.
Note that for each injectivity point \(x \in M\), the \(\Z_2\)-mapping degree of \(f\) satisfies \(\deg_{\Z_2}(f) = 
\deg_{\Z_2}(f \rvert_x)\), where \(\deg_{\Z_2}(f \rvert_x)\) denotes the local mapping degree of \(f\) at the point \(x\).
Moreover, if \(x \in M\) is a point where \(f\) is differentiable and \(\dd_x f\) is invertible, then \(x\) is a discrete point in the fiber \(f^{-1}(f(x))\) and \(\deg_{\Z_2}(f \rvert_x) = 1\).
Since \(\lvert \det(\dd_x f) \rvert = 1\) almost everywhere on \(M\), the set of such points also has full measure.
Finding an injectivity point \(x \in M\) where \(f\) is differentiable with \(\dd_x f\) invertible, we conclude
\[
   \deg_{\Z_2}(f) = \deg_{\Z_2}(f\rvert_x) = 1.
\]
In particular, \(f\) is surjective.

It remains to verify injectivity.
Since both \(M\) and \(N\) are compact and without boundary, we can find an \(r_0 > 0\) such that for any \(r \leq r_0\), \(x \in M\) and \(y \in N\), we have 
\[
\vol_M(\Ball_r(x)) \geq \tfrac{3}{4} r^n \omega_n, \qquad
\vol_N(\Ball_r(y)) \leq \tfrac{4}{3} r^n \omega_n.
\]

Denote the Lipschitz constant of \(f\) by \(L\).
Let $x_1,\dots,x_N\in M$ be pairwise distinct preimages of a point $y\in N$ under $f$.
Choose \(0 < r \leq \frac{r_0}{L}\) such that the \(\Ball_r(x_i)\) are pairwise disjoint.
For \(i\neq j\) it follows that 
\begin{equation}
\vol_N(f(\Ball_r(x_i)) \cap f(\Ball_r(x_j)))=0
\label{eq.almost_disjoint}
\end{equation}
since \(\# f^{-1}(y) \leq 1\) for almost all \(y \in N\).
Moreover, since \(f\) is \(L\)-Lipschitz, we have \(f(\Ball_r(x_i)) \subseteq \Ball_{Lr}(y)\).
Therefore, we have
\begin{align*}
\tfrac{3N}{4} r^n \omega_n
&\le
\sum_{i=1}^N \vol_M(\Ball_r(x_i)) \\
&=
\sum_{i=1}^N \vol_N(f(\Ball_r(x_i))) \\
&=
\vol_N\bigg(f\Big(\bigcup\nolimits_{i=1}^N \Ball_r(x_i)\Big)\bigg) \\
&\le
\vol_N(\Ball_{Lr}(y)) \\
&\le 
\tfrac43 L^nr^n\omega_n.
\end{align*}
Hence \(N\le \tfrac{16}{9} L^n\).
In particular, each \(y\in N\) has only finitely many preimages under \(f\).

Since the $\Z_2$-mapping degree of \(f\) is nontrivial and coincides with the sum of the local $\Z_2$-mapping degrees of the preimages, there must be a preimage \(x \in M\) of \(y\) such that the local mapping degree at \(x\) is nontrivial.
In particular, \(f(\Ball_r(x))\) contains an open neighborhood of \(y\), i.e.\ \(\Ball_\rho(y)\subseteq f(\Ball_r(x))\) for some \(\rho>0\).
Assume that there is another preimage \(x'\in M\) of \(y\). 
Choose an injectivity point $x''\in \Ball_{\rho/2L}(x')$.
Then we have \(f(x'')\in \Ball_{\rho/2}(y)\) and the local $\Z_2$-mapping degree of $f$ at $x''$ is nontrivial.
Hence $f(\Ball_{\rho/2L}(x''))$ contains an open neighborhood of $f(x'')$.
Note that \(f(\Ball_{\rho/2L}(x''))\subseteq \Ball_{\rho/2}(f(x''))\subseteq \Ball_\rho(y)\), see Figure~\ref{fig.1}.
Hence, \(f(D_r(x))\cap f(D_r(x'))\) contains a non-empty open set, contradicting \eqref{eq.almost_disjoint}.

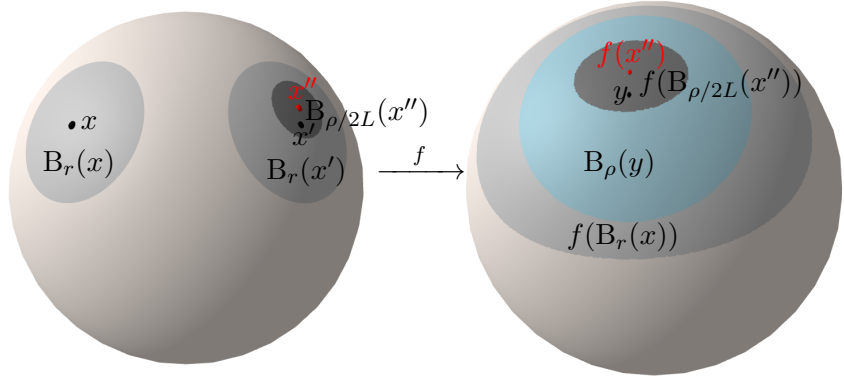
\begin{figure}[h]
\centering
\begin{overpic}[width=0.8\textwidth]{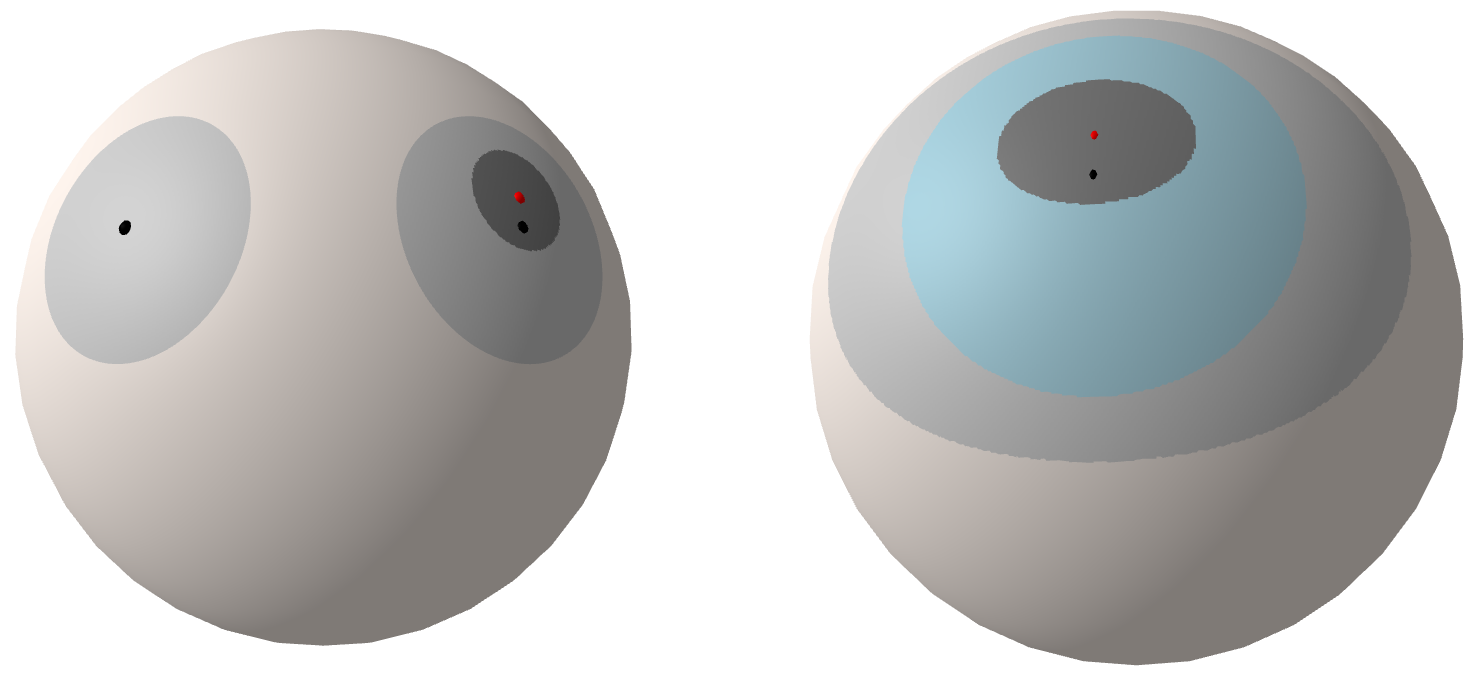}
\put(9.5,30){$x$}
\put(34.5,28){$x'$}
\put(34,33.5){$\textcolor{red}{x''}$}
\put(72.5,33.5){$y$}
\put(70.5,38){$\textcolor{red}{f(x'')}$}
\put(44.5,24){$\xrightarrow{\quad f\quad}$}
\put(5,25){$\Ball_r(x)$}
\put(31.5,24){$\Ball_r(x')$}
\put(36,31){$\Ball_{\rho/2L}(x'')$}
\put(67,16.2){$f(\Ball_{r}(x))$}
\put(69,25){$\Ball_{\rho}(y)$}
\put(75.5,34.5){$f(\Ball_{\rho/2L}(x''))$}
\end{overpic}
\caption{Two potential preimages \(x\) and \(x'\) of the same point \(y\in N\) under \(f\).}
\label{fig.1}
\end{figure}

This shows that \(f\) is injective.
Every bijective continuous map between compact manifolds is a homeomorphism.
\cref{lip_homeomorphism} now shows that $f^{-1}$ is $L^{n-1}$-Lipschitz.
In particular, if \(f\) is \(1\)-Lipschitz, then so is \(f^{-1}\) and hence \(f\) is an isometry.
\end{proof}

\begin{proof}[Proof of \cref{Ricci_volume_llarull}]
   Since \(\Ric \geq n-1\), the Bishop--Gromov volume comparison theorem (see e.g.~\cite{Gallot-Hulin-Lafontaine}*{Theorem~4.19}) yields \(\vol(M) \leq \vol(\Sphere^{n})\).
   \Cref{1lip_injectivity} now implies that \(f\) is a bi-Lipschitz homeomorphism and that $\vol(M)=\vol(\Sphere^{n})$.
   Thus, \(M\) is isometric to \(\Sphere^n\) by the equality discussion of Bishop--Gromov volume comparison~\cite{Gallot-Hulin-Lafontaine}*{Theorem~4.20}.
   Moreover, if \(f\) is \(1\)-Lipschitz, it is an isometry, again by \cref{1lip_injectivity}.
\end{proof}

\section{Proof of Theorem~\ref{thmB}}
\label{sec.B}

We start with the construction of the metric on $M$.

\begin{lemma}\label{existence_of_the_wurst}
   Let \(M\) be a closed connected $n$-dimensional manifold which admits a Riemannian metric of positive scalar curvature with $n\ge3$.

   Then for each \(S_0 > 0\) there exists a constant \(\delta(M,S_0) > 0\), such that for each $\delta\in(0,\delta(M,S_0)]$ and each \(l > 0\) there exists a Riemannian metric \(g_M\) on \(M\) with \(\scal_{g_M} \geq S_0\) and an isometric embedding \([0,l] \times \Sphere^{n-1}_\delta \hookrightarrow (M, g_M)\). 
   Moreover, we can ensure that \(M \setminus ([0,l] \times \Sphere^{n-1}_\delta)\) is disconnected.
\end{lemma}

\begin{proof}
   We fix a metric $g'$ on the open $n$-ball $\Disk^n$ such that $(\Disk^n, g')$ contains an isometric copy of $[0,\infty)\times \S^{n-1}$ and $\scal_{g'}>0$.
   This is possible because $n\ge3$.
   Pick a point $p\in \Disk^n\setminus [0,l']\times \S^{n-1}$.
   Furthermore, fix a metric $g_0$ on $M$ with $\scal_{g_0} > 0$ and choose a point $q\in M$.

   We perform Gromov-Lawson surgery \cite{GL}*{Theorem~A} to obtain a Riemannian metric $g_1$ on $M\sharp \Disk^n=M\setminus \{\text{point}\}$ with $\scal_{g_1} > 0$ which differs from $g_0$ and $g'$ only in a small neighborhood of $q$ and $p$, respectively.
   In particular, $(M\setminus \{\text{point}\},g_1)$ still contains an isometric copy of $[0,\infty)\times \S^{n-1}$, see \cref{fig.5}.
\begin{figure}[h]
\centering
\begin{overpic}[width=0.6\textwidth]{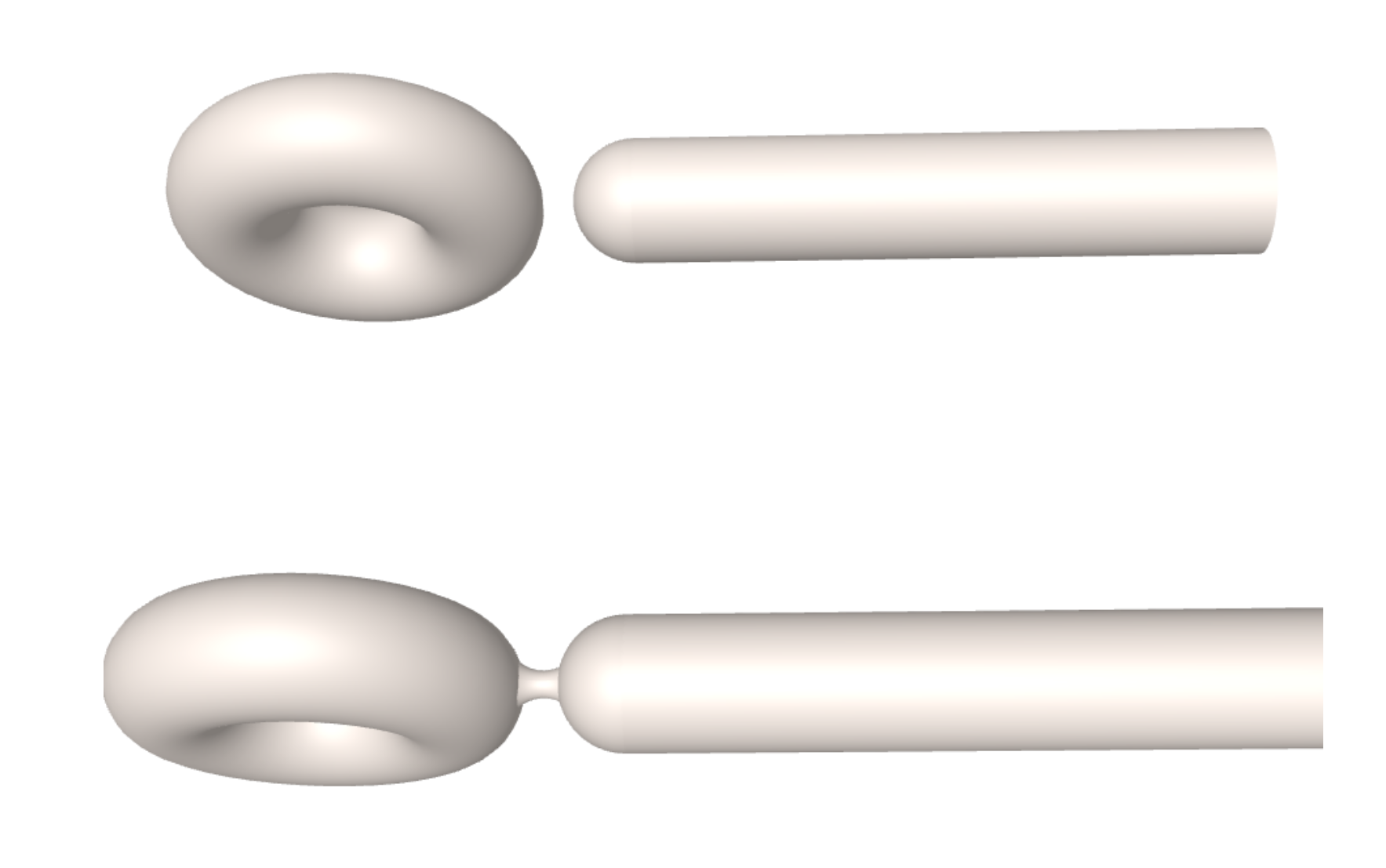}
\put(20,33){$(M,g_0)$}
\put(60,39){$(\Disk^n,g')$}
\put(25,0){$(M\setminus\{\text{point}\},g_1)$}
\end{overpic}
\caption{Construction of the metric $g_1$}
\label{fig.5}
\end{figure}

   For sufficiently small $\delta>0$, the metric $g_\delta=\delta^2 g_1$ has scalar curvature $\scal_{g_\delta} \geq S_0$ and $(M\setminus \{\text{point}\},g_\delta)$ contains an isometric copy of $[0,\infty)\times \S^{n-1}_\delta$.
   Given $l>0$ we can cut the cylinder $[0,\infty)\times \S^{n-1}_\delta$ at $l$.
\begin{figure}[h]
\centering
\begin{overpic}[width=0.6\textwidth]{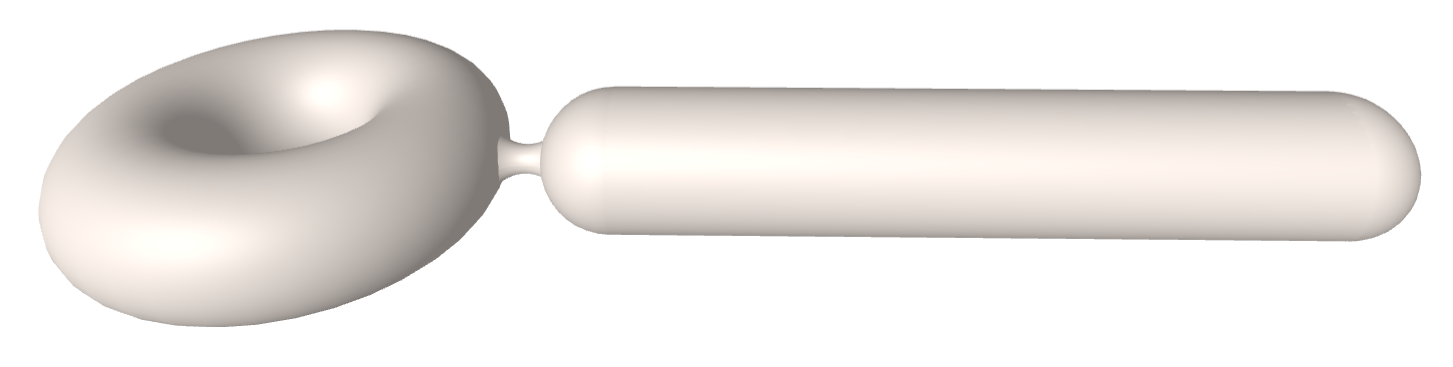}
\put(41,3){$|\!\raisebox{2.5pt}{\rule{20mm}{.6pt}}\,l\,\raisebox{2.5pt}{\rule{20mm}{.6pt}}\!|$}
\end{overpic}
\caption{Capping off the half-infinite cylinder}
\label{fig.6}
\end{figure}

The remaining cylinder $[0,l]\times \S^{n-1}_\delta$ can be capped off at $\{l\}\times \S^{n-1}_\delta$ to yield the desired metric $g_M$ on $M$, see \cref{fig.6}.
\end{proof}

\begin{lemma}\label{projection_of_the_wurst}
   Let \(M\) be a closed connected $n$-dimensional manifold which admits a Riemannian metric of positive scalar curvature with $n\ge3$.

   Then for each \(S_0 > 0\) there exists a constant \(\delta(M,S_0) > 0\), such that for each $\delta\in(0,\delta(M,S_0)]$ and each \(l > 0\) there exists a Riemannian metric \(g_M\) on \(M\) with \(\scal_{g_M} \geq S_0\) and a smooth surjective $1$-Lipschitz map \(f \colon M \to [0,l]\times\Disk^{n-1}_\delta\).
   Here \([0,l]\times\Disk^{n-1}_\delta\) carries the flat product metric.
\end{lemma}

\begin{proof}
According to \cref{existence_of_the_wurst}, given $S_0$, $\delta$ and $l$, we can find a Riemannian metric $g_M$ on $M$ with $\scal_{g_M} \geq S_0$ such that $M$ is isometric to 
\[
M_-\cup_{\{-3\}\times\S^{n-1}_\delta} \Big([-3,l+3]\times\S^{n-1}_\delta\Big)\cup_{\{l+3\}\times\S^{n-1}_\delta} M_+ \,,
\] 
where $M_+$ and $M_-$ are compact Riemannian manifolds with boundary isometric to $\S^{n-1}_\delta$.
The orthogonal projection $\R^n\to \{0\}\times\R^{n-1}=\R^{n-1}$ restricts to a smooth surjective $1$-Lipschitz map $\pi_\delta\colon\S^{n-1}_\delta\to\Disk^{n-1}_\delta$.
Choose a smooth function $\varphi\colon\R\to\R$ with the properties
\begin{enumerate}[label=$\triangleright$]
\item \(0\le\varphi\le 1\) everywhere,
\item \(\varphi = 0\) outside \([-3,l+3]\),
\item \(\varphi = 1\) on \([-1,l+1]\),
\item \(\lvert\dot{\varphi}\rvert < 1\) everywhere,
\end{enumerate}
and a smooth function \(\psi\colon\R\to\R\) satisfying
\begin{enumerate}[label=$\triangleright$]
\item \(0\le\psi\le l\) everywhere,
\item \(\psi = 0\) on \((-\infty,-1]\),
\item \(\psi = l\) on \([l+1,\infty)\),
\item \(\lvert\dot{\psi}\rvert < 1\) everywhere.
\end{enumerate}
Then the map
\begin{gather*}
f(x)=
\begin{cases}
(\psi(t),\varphi(t)\pi_\delta(y)), &\text{if } x=(t,y)\in [-3,l+3]\times\S^{n-1}_\delta, \\
(0,0), &\text{if } x\in M_-, \\
(l,0), &\text{if } x\in M_+,
\end{cases}
\end{gather*}
does the job.
\end{proof}

\begin{proof}[Proof of \cref{thmB}]
If suffices to consider the case $\eps=2$ because for arbitrary $\eps>0$ one can then apply the $(\eps=2)$-result with scalar curvature bound $S_0\cdot(2\eps)^{-2}$ and then rescale the metric on $M$.
In view of \cref{projection_of_the_wurst}, we need to construct a smooth surjective $2$-Lipschitz map $[0,l]\times \Disk^{n-1}_\delta\to N$ for some $\delta>0$ and $l>0$.

Let $\delta(M,S_0)$ be as in \cref{projection_of_the_wurst}.
Choose $\delta\in(0,\delta(M,S_0)]$ so small that $\lvert \dd_v\exp_y\rvert < 2$ whenever $v\in T_yN$ with $\lvert v\rvert\le\delta$.
We choose a smooth curve $\gamma\colon [0,l_1]\to N$ parametrized by arc-length which is $\delta$-dense in $N$, meaning that the $\delta$-tube about the trace of $\gamma$ covers all of $N$, i.e., $\Ball_\delta(\gamma([0,l_1]))=N$.

We choose an orthonormal frame $e_1,\dots,e_{n-1}$ of $\dot{\gamma}(0)^\perp\subseteq T_{\gamma(0)}N$ and parallel translate it along $\gamma$ w.r.t.\ the normal connection.
We obtain an orthonormal frame $e_1(t),\dots,e_{n-1}(t)$ of $\dot{\gamma}(t)^\perp \subseteq T_{\gamma(t)N}$ for all $t\in[0,l_1]$.
We fix a constant $\Lambda > 0$ to be chosen later and consider the map
\[
h\colon [0,\Lambda l_1]\times \Disk_\delta^{n-1}\to N, 
\quad 
(t,x)\mapsto \exp_{\gamma(t/\Lambda)}\bigg(\sum_{i=1}^{n-1} x_i e_i(t/\Lambda)\bigg).
\]
The map $h$ is smooth and surjective.
The derivatives in the $x_i$-directions are bounded by $2$ because of the choice of $\delta$.
The same holds for the derivative in the $t$-direction if we choose $\Lambda$ large enough.
Thus $h$ is $2$-Lipschitz as desired.
\end{proof}


\end{document}